\documentclass{pspum-l}
\usepackage{amsmath,textcomp,latexsym,graphics}
\usepackage{amssymb,amsthm}

\def\R{\mathbb{R}}

\numberwithin{equation}{section}

\newcommand{\be}{\begin{equation}}
\newcommand{\ee}{\end{equation}}
\newcommand{\ba}{\begin{array}}
\newcommand{\ea}{\end{array}}

\newcommand{\bpm}{\begin{pmatrix}}
\newcommand{\epm}{\end{pmatrix}}

\newcommand{\bd}{\begin{definition}}
\newcommand{\ed}{\end{definition}}

\newtheorem{theorem}{Theorem}[section]

\newtheorem{lemma}[theorem]{Lemma}
\newtheorem{cor}[theorem]{Corollary}

\theoremstyle{remark}
\newtheorem{remark}{Remark}

\theoremstyle{definition}
\newtheorem{definition}{Definition}[section]

\begin{document}

\title{On a transformation of Bohl and its discrete analogue}
\author{Evans M. Harrell II}
\address{School of Mathematics,
Georgia Institute of Technology,
Atlanta, GA 30332-0610 USA}
\email{harrell@math.gatech.edu}
\author{Manwah Lilian Wong}
\address{School of Mathematics,
Georgia Institute of Technology,
Atlanta, GA 30332-0610 USA}
\email{wong@math.gatech.edu}



\begin{abstract}
Fritz Gesztesy's varied and prolific career has 
produced many transformational 
contributions to the spectral theory of one-dimensional Schr\"o\-dinger equations.  
He has often done this by revisiting the insights of great 
mathematical analysts of the past, connecting them in 
new ways, and reinventing them in a thoroughly modern context.

In this short note we recall and relate some 
classic transformations that figure among 
Fritz Gestesy's favorite tools of spectral theory, 
and indeed thereby make connections
among some of his favorite scholars of the past, 
Bohl, Darboux, and Green.  After doing this in the context of 
one-dimensional Schr\"odinger equations on the line,
we obtain some novel analogues for discrete
one-dimensional Schr\"odinger equations.
\smallskip

Dem einzigartigen Fritz gewidmet.
\end{abstract}
\maketitle

\section{Introduction}

In 1906 \cite{Boh}, Bohl introduced a nonlinear transformation for solutions of Sturm-Liouville equations, which is an exact, albeit implicit, counterpart to the Liouville-Green aproximation \cite{Olv}.  Bohl used the transformation as a tool in oscillation theory, and this has continued to be the main use of the Bohl transformation in the hands of later authors.  Notably, R\'ab \cite{Rab} used the Bohl transformation to prove necessary and sufficient conditions for oscillation of solutions, and showed that it is an effective foundation for Sturm-Liouville oscillation theory.  Willett's lecture in \cite{Wil} provides a clear description in English of the Bohl transformation in oscillation
theory, including the contributions of R\'ab, while Reid's monograph \cite{Re80} 
compares and contrasts it with the Pr\"ufer transformation.  
See also \cite{Re71,GeSiTe,GeUn,GoSM1,GoSM2}.

In \cite{DaHa} \S 4, Davies and Harrell introduced a non-oscillatory variant of the Bohl transformation 
to connect the notions of Liouville-Green approximation, Green functions, and Agmon metrics for exponential decay of solutions.  Some spectral bounds were derived as consequences.  This analysis was extended
in a series of articles by Chernyavskaya and Shuster 
(e.g., \cite{Che,ChSh99,ChSh07}), to address questions of solvability, regularity, 
estimates of Green functions,
and asymptotics in Sturm-Liouville theory.

In this note we begin with a largely expository treatment of the classic Bohl transformation, 
concentrating for simplicity on the situation where all coefficients are real and 
regular and the Sturm-Liouville equation is in the standard form
of the one-dimensional Schr\"odinger equation.  Then in the last section 
we show how the technique can be adapted to the case of a discrete Schr\"odinger 
equation on the integers.

\section{The interplay of the Bohl and Green functions}\label{B&G}

Let $V$ be real-valued and continuous, and consider a solution basis 
for the Sturm-Liouville equation
\begin{equation}\label{SLE}
-u^{\prime\prime} + V(x) u = 0;
\end{equation}
we may normalize the basis $u_{1,2}(x)$
so that $W[u_1,u_2] := u_1 u_2^\prime - u_2 u_1^\prime = 1$.  
(There is no assumption of an eigenvalue 
$0$.  For our purposes a possible nonzero spectral parameter has simply been incorporated into $V$.)
The Bohl transformation maps this solution basis onto a second solution basis
with remarkable properties, some of which are collected 
in a nutshell version in this section.

\begin{definition}\label{BohlDef}
Given a solution basis 
$\{u_{1,2}(x)\}$ of \eqref{SLE},
chosen so that the Wronskian 
$W[u_1,u_2] := u_1 u_2^\prime - u_2 u_1^\prime = 1$, we
define the {\it diagonal function} by
\begin{equation}\label{ZDef}
Z[u_1,u_2](x) := (u_1(x)u_2(x))^{1/2}.
\end{equation}
The {\it Bohl transformation} of $\{u_{1,2}(x)\}$
is an equivalent solution basis of \eqref{SLE},
defined in terms of $\{u_{1,2}(x)\}$ by
\begin{equation}\label{phipm}
{\mathcal B}:\{u_1(x), u_2(x)\} \to
\left\{\phi^{\pm}(x) := Z(x) \exp{\left(\pm \int_{}^x{\frac{1}{2 Z^2(t)} dt}\right)}\right\}.
\end{equation}
\end{definition}

\begin{remark}
a)  The choice of the complex 
phase of the square root in \eqref{ZDef} is unimportant, but should be continuous in $x$.  
For brevity we write $Z(x)$ for $Z[u_1,u_2](x)$ when the dependence on $u_{1,2}$ is clear.  
The reason for calling it the diagonal function is that, as will be seen below,
$$
Z^2(x) = G_0(x,x),
$$
where $G_0(x,x)$ is the diagonal of a certain Green function $G_0(x,y)$
for \eqref{SLE}.  Among the useful properties of the function $Z$ 
is that it solves the {\it diagonal differential equation}
\begin{equation}\label{ZDE}
{\mathcal J}[Z] := -Z^{\prime\prime} + V(x) Z - \frac{1}{4 Z^3} = 0,
\end{equation}
\cite{DaHa,GoSM1,GoSM2,Wil}.
\medskip

\noindent
b).  In fact, with the oscillatory situation in mind
Bohl originally wrote the  
solution basis in the Liouville-Green form
$$
\frac{1}{\sqrt{R}} \sin\left(\int^x{R(t) dt} \right)
$$
and
$$
\frac{1}{\sqrt{R}} \cos\left(\int^x{R(t) dt} \right),
$$
which is equivalent to \eqref{phipm}
under the identification $2 Z^2 \to i/R$
and some harmless linear combinations.

\noindent
c).  We recall that Gesztesy and Simon \cite{GeSi}
have made connections
between the Krein spectral shift function, the related Xi function,
and the diagonal of the Green function.
\end{remark}

Calling upon 
\cite{DaHa} \S 4,
we collect some facts, which are verifiable directly:

\begin{theorem}\label{BTfacts}\hfill
\begin{enumerate}
\item
If $u_1(x)$ and $u_2(x)$ are solutions to \eqref{SLE}
such that $W[u_1,u_2] = 1$, and
$u_1(x) u_2(x)$ does not vanish on the interval 
$(a,b)$, then 
$Z[u_1,u_2](x)$ satisfies \eqref{ZDE} on $(a,b)$.
\item
If $Z$ is a nonvanishing solution of 
\eqref{ZDE} on $(a,b)$, then 
$\phi^{\pm}(x)$ as defined in \eqref{phipm}
provide a pair of independent solutions of \eqref{SLE} on $(a,b)$.
In particular, each $\phi^{\pm}(x)$ is a linear combination of $\left\{u_{1,2}\right\}$
and vice versa.
\item
If $x_{>,<} := \max(x,y)$, resp. $\min(x,y)$, then
\begin{equation}\label{Gfn}
G_0(x,y) := Z(x)Z(y) \exp{\left(- \int_{x_<}^{x_>}{\frac{1}{2 Z^2(t)} dt}\right)}
\end{equation}
is a Green function for \eqref{SLE}, in the sense that
\begin{equation}
\left(-\frac{\partial^2}{\partial x^2} + V(x)\right) G_0(x,y) = \delta(x-y).
\end{equation}
\end{enumerate}
\end{theorem}

As an integral kernel, $G_0$ defines the inverse of a particular realization of $- \frac{d^2}{dx^2} + V$, but not {\em a priori} one for which the domain of definition includes 
$u_1$ or $u_2$, because of a possible mismatch of boundary conditions at finite points.  This issue is not important 
for questions of oscillation or asymptotic behavior at infinity, but another concern remains, 
namely the possibility that $Z$ vanishes,
which would invalidate the transformation.  Because of this we recall that in the absence of imposed
finite boundary conditions, complex solutions can always be used to
prevent $Z$ from vanishing: 

\begin{lemma}\label{nonvan}
Suppose that $u$ is a solution of \eqref{SLE} on a 
finite or infinite interval $(a,b)$, and that at some $x_0 \in (a,b)$, ${\rm Re}(u(x_0)) {\rm Im}(u(x_0)) \ne 0$ and 
$u^\prime(x_0)/u(x_0) \notin \R$.  Then $u$ does not vanish on $(a,b)$.
\end{lemma}

\begin{proof}
Because $V$ is real-valued, ${\rm Re} \, u$ and ${\rm Im} \, u$ each satisfy \eqref{SLE}, and it therefore suffices to show that they are independent.

Letting $\, \alpha = u^\prime(x_0)/u(x_0)$, a calculation shows that.
\begin{align*}
&\frac{{\rm Re} \, u^\prime(x_0)}{{\rm Re} \, u(x_0)} = {\rm Re} \, \alpha - {\rm Im} \, \alpha \frac{{\rm Im} \, u(x_0)}{{\rm Re} \, u(x_0)},\\
&\frac{{\rm Im} \, u^\prime(x_0)}{{\rm Im} \, u(x_0)} = {\rm Re} \, \alpha + {\rm Im} \, \alpha \frac{{\rm Re} \, u(x_0)}{{\rm Im} \, u(x_0)}.
\end{align*}
It follows that
$$
\frac{{\rm Re} \, u^\prime(x_0)}{{\rm Re} \, u(x_0)} - \frac{{\rm Im} \, u^\prime(x_0)}{{\rm Im} \, u(x_0)} 
= - {\rm Im} \, \alpha\left(\frac{{\rm Im} \, u(x_0)}{{\rm Re} \, u(x_0)} + \frac{{\rm Re} \, u(x_0)}{{\rm Im} \, u(x_0)} \right),
$$
and therefore
$$
W[{\rm Im} \, u, {\rm Re}\, u] = - {\rm Im} \, \alpha \left(({\rm Re} \, u(x_0))^2 + ({\rm Im} \, u(x_0))^2 \right) \ne 0. 
$$
\end{proof}

This standard lemma implies that given any two linearly independent solutions of \eqref{SLE} it is always possible to find 
a pair of complex-valued linearly independent combinations that are non-vanishing on $(a,b)$.  The Wronskian of the 
new pair may be set to 1 by
multiplying one solution by an appropriate constant,
justifying the conclusions of Theorem \ref{BTfacts}.  
It is of some use to consider a particular $Z$ determined as follows.

By construction any solution $u$ as set forth in Lemma \ref{nonvan} will be linearly independent of its complex conjugate.  We may therefore choose a complex number $\alpha$ so that $W[u, \alpha^2 \overline{u}] = 1$, 
which ensures that 
\begin{equation}\label{specialZ}
Z(x) = \alpha |u(x)|.
\end{equation}
Indeed, $\arg(\alpha)$ is restricted 
to the values $k\pi/4$ for integer $k$,
as can be seen for example from \eqref{ZDE}, which implies that
\begin{equation}
- |u|^{\prime\prime} + V |u| = \frac{1}{4 \alpha^4 |u|^3} \in \R.
\end{equation}
(In passing we note the implication that the expression on the left does not change sign.)  In fact, there are only 
two truly distinct cases for $\arg(\alpha)$,  {\em viz}.,\ $0$ and $\frac{\pi}{4}$, 
due to the simple scalings in \eqref{ZDE} and 
Theorem~\ref{BTfacts} when $Z$ is replaced by $i Z$.
If $\alpha > 0$, i.e., $Z(x) > 0$, then the solutions $\phi^\pm$ do not change sign, 
which corresponds to the case of {\em disconjugacy}
for the ODE \eqref{SLE} in the classical theory \cite{Hart,Re71}.   This situation 
was the focus of \cite{DaHa}, in which some spectral bounds were derived and it was argued 
that $1/2Z^2$ defines an Agmon metric.

Otherwise, it may be assumed without of loss of generality that $\arg(\alpha) = \frac{\pi}{4}$, 
for which solutions may oscillate, in that their arguments increase or decrease by
$n \pi$ for $n > 1$ as $x \to \infty$.  A central question of Sturmian theory
is whether solutions oscillate infinitely often, or only finitely often.  When the
increase in the argument of a solution is infinite, the equation \eqref{SLE}
is said to be {\em oscillatory}.
An approach to oscillation theory, equivalent to that of 
\cite{Rab,Wil} but bringing out the role of Green functions, 
can be based on the following version of a result of Gagliardo,
as cited in \cite{Wil}:

\begin{cor}
[Cf. \cite{Wil}, Corollary~3.2.]
Suppose that \eqref{SLE} holds on an infinite interval $(a, \infty)$, and let $G_B(x,y)$ be the Green function
constructed according to the prescription
leading to \eqref{specialZ}.
Then either $G_B(x,y) \in \R$ for all $x,y$, 
in which case the solution basis $\phi^\pm$ is nonoscillatory, or else:
The phase of $\phi^\pm$ has only finite increase on $(a,\infty)$ iff $1/G(x,x) \in L^1(a,\infty)$.
\end{cor}

As has been known since the work of Wigner and von Neumann and Wigner,
it is possible for eigenvalues to be embedded in the continuous spectrum of
Sturm-Liouville equations \cite{NeWi,ReSi}.  This phenomenon requires oscillatory solutions 
to be square-integrable, and can thus be related to the Bohl transformation as follows:

\begin{cor}
[R\'ab]
An oscillatory solution of \eqref{SLE}, written in the form
$$
Z(x) \exp{\left(\pm \frac{i}{2} \int_a^x{\frac{1}{|Z(t)|^2} dt}\right)},
$$
exists and is square integrable if and only if the nonlinear equation
$$
- w^{\prime\prime}(x) + V(x) w(x) = - \frac{1}{4 w^3(x)}
$$
has a square-integrable solution.
\end{cor}

We close this section with a Darboux-type factorization
\cite{Dar,GeTe1}, the novel feature of which is 
the role played by the diagonal function
and the Bohl solution basis
\eqref{phipm}.
For any complex valued, nonvanishing function
$Z(x) \in AC^1[(a,b)]$, define
$$
{\mathcal D}^\pm[Z] := \frac{d}{dx} - \frac{Z^\prime}{Z} \mp \frac{1}{2 Z^2}.
$$
It is immediate to see that
$$
{\mathcal D}^\pm[Z] \phi^\pm = 0,
$$
where $\phi^\pm$ are defined in terms of $Z$ by \eqref{phipm}.  
A further calculation reveals that
\begin{equation}\label{Darboux}
\left({\mathcal D}^\pm[Z] - 2 \frac{d}{dx}\right) {\mathcal D}^\pm [Z]=
-\frac{d^2}{dx^2} + \frac{Z^{\prime\prime}}{Z} + \frac{1}{4 Z^4} = -\frac{d^2}{dx^2} 
+ V(x),
\end{equation}
provided that $Z$ satisfies the diagonal differential equation \eqref{ZDE}.
An alternative way to express the two factorizations in \eqref{Darboux} is that
$$
-\frac{d^2}{dx^2} 
+ V(x) = 
\left({\mathcal D}{\Big[}\overline{Z}{\Big]}^\mp\right)^* {\mathcal D}[Z]^\pm,
$$
where $^*$ designates the formal adjoint operation.

\section{The discrete form of the Bohl ${\mathcal J}[Z]$}

In this section we show that most of the transformations and relationships presented in 
the first section have counterparts for discrete one-dimensional Schr\"odinger 
equations.
(Part of the material in this section has appeared in a 
preprint \cite{HaWo}, which has been expanded and divided
for publication as two articles.)
some details are
rather different from the continuous case, making it uncertain how far the analogy goes,
especially in the oscillatory case.
A full-fledged oscillation theory for discrete problems based
on an analogue of the Bohl transformation and its connection
to Green functions would be an interesting next project.

Let $\Delta$ denote the discrete second-difference operator on 
the positive integers.  We standardize the Laplacian such that
$(\Delta f)_n := f_{n+1} + f_{n-1} - 2 f_n$ for $f = (f_n) \in \ell^2(\mathbb{N})$,
and consider equations of the form 
\begin{equation}\label{one}
(-\Delta+V)u  =  0, 
\end{equation}
where the potential-energy 
function $V$ is a diagonal operator with 
real values $V_n$. 

Eq.~\eqref{one} and its solutions share many of the properties of 
classical Sturm-Liouville equations, as is laid out for example in \cite{Aga}.  
For our purposes we recall that:  The solution space is two-dimensional, and the Wronskian of any two solutions
\begin{equation}\label{Wdef} 
W[u^{(1)}, u^{(2)}] := u^{(1)}_n u^{(2)}_{n+1} - u^{(1)}_{n+1} u^{(2)}_n
\end{equation}
is constant.   
A {\it Green matrix} as a solution of
\begin{equation}\label{GreenDef}
(-\Delta+V) G  =  I, 
\end{equation}
where $I$ is the identity matrix, and every Green matrix can be written as 
the sum of a vector in the null space of $(-\Delta+V)$ and the particular Green matrix
\begin{equation}
G_{m,n}^{(p)} := \frac{u_{\max(m,n)}^{(1)} u_{\min(m,n)}^{(2)}}{W[u^{(1)}, u^{(2)}]},
\end{equation}
provided that $\{u^{(1)},u^{(2)}\}$ are linearly independent.

A feature of the discrete Schr\"odinger equation 
\eqref{one} 
that is not shared by \eqref{SLE} is 
an invariance under the transformation
\begin{eqnarray}\label{sym}
u_n \to (-1)^n u_n\notag\\
V_n \to - 4 - V_n,
\end{eqnarray}
as can be easily checked.
Among other things, this implies that
any fact proved under the assumption, for example, that $V_n > 0$
has a counterpart for $V_n < -4$, with systematic sign changes.


Our goal in this section is to present an analogue of the Bohl transformation for the discrete Schr\"odinger equation \eqref{one}.  In particular, we offer a discrete version of
some of the results of
\cite{DaHa}, \S 4, and show in particular that
the diagonal elements $G_{nn}$ of the Green matrix allow 
the full solution space to be recovered formulaically.
We build on some earlier
steps in this direction 
by Chernyavskaya and Shuster \cite{ChSh00,ChSh01}.  As in \cite{DaHa}
we furthermore point out connections between the diagonal of the Green
matrix and an Agmon distance for \eqref{one}.

In the discrete situation the use of exponentials of integrals is not the most natural, 
so we instead seek to represent a pair of solutions in the forms
\begin{equation}\label{prodansatz}
\varphi^+_n=z_n\prod^n_{\ell=1}S_\ell, \quad\quad\quad
\varphi^-_n=z_n\left(\prod^n_{\ell=1}S_\ell\right)^{-1}.
\end{equation}
Since the product of these two solutions 
is the diagonal of a Green matrix, up to a constant multiple, 
this suggests that if we begin by selecting a Green matrix 
such that $G_{nn}$ is nonvanishing, then we can directly define

$$
z_n := (G_{nn})^{1/2}.
$$
It remains to work out the most convenient 
form of $S^{\pm 1}_\ell$.  If the Wronskian is scaled so that
$W[\varphi^-, \varphi^+] = 1$, then substitution of the ansatz \eqref{prodansatz}
leads after a calculation to 

\begin{equation}\label{SCond}
S_n - \frac{1}{S_n} = \frac{1}{z_n z_{n-1}}.
\end{equation}

Here we pause to observe two ambiguities in relating 
$\varphi_n^{\pm}$ to the potential $V$.  The first is that, due to 
the invariance \eqref{sym}, if 
\begin{equation}
G_{mn} = \psi_{\min(m,n)}^+ \psi_{\max(m,n)}^-
\end{equation}
is the Green matrix for some potential function $V_n$, then 
the same diagonal elements $G_{nn}$ also belong to the Green
matrix for an equation of type \eqref{one} but with potential function
$\widetilde{V}_n = -4 - V_n$.  
Secondly, \eqref{SCond} is equivalent to a quadratic expression for 
$S_n$, and therefore the solution is generally nonunique.  
These ambiguities are avoided when 
the Schr\"odinger operator 
$H = -\Delta + V$
in \eqref{one} is positive, 
so $G_{nn} > 0$ and by convention $z_n > 0$.  We can then fix 
$S_n$ as the larger root of \eqref{SCond}.  

Accordingly, in this situation 
we simply define
\begin{equation}\label{Sfromz}
S_n^{[z]}:=\frac{1+\sqrt{1+ 4 z_n^2 z_{n-1}^2}}{2z_nz_{n-1}}.
\end{equation}
A pair of functions $\varphi^{\pm}_n$ can now be defined by 
the ansatz \eqref{prodansatz},
i.e., when expressed in terms of $z_n$,
\begin{align}\label{phifromz}
&\varphi^{+}_n:=z_n\prod^n_{k=m+1}\left(\frac{1+\sqrt{1+ 4 z_n^2 z_{n-1}^2}}{2z_nz_{n-1}}\right),\\
&\varphi^-_n=z_n\left(\prod^n_{k=m+1}\left(\frac{1+\sqrt{1+ 4 z_n^2 z_{n-1}^2}}{2z_nz_{n-1}}\right)\right)^{-1}.\nonumber
\end{align}

\noindent
Remarkably, with this definition,
both $\varphi^+$ and $\varphi^-$ solve a single equation
of the form \eqref{one}, where the potential function $V_n$
is determined from $z_n$ via
\begin{align}\label{Vfromz}
V^{[z]}_n&:=\frac{\Delta \varphi^+_n}{\varphi^+_n}=\frac{
1+\sqrt{1+4 z_n^2 z_{n+1}^2}}{2z^2_n}+
\frac{2z^2_{n-1}}
{1+\sqrt{1+4 z_n^2 z_{n-1}^2}} - 2 \nonumber\\
&= \frac{z_{n+1}}{z_n} S_{n+1}^{[z]}+\frac{z_{n-1}}{z_n S_n^{[z]}} - 2,
\end{align}
provided that $V_n > -2$.
(Else a different root must be chosen in 
\eqref{Sfromz}.)  
To see that $\varphi_n^\pm$ solve the same discrete Schr\"odinger equation, 
let us separately calculate
\begin{equation}\label{phiminusstuff}
\frac{\Delta\varphi^-_n}{\varphi^-_n}=\frac{z_{n+1}}{z_n S_{n+1}^{[z]}}+
\frac{z_{n-1}}{z_n}S_n^{[z]} - 2,
\end{equation}
and note that since $S_n^{[z]}$ has been chosen to satisfy \eqref{Sfromz}, 
the difference between these last 
two expressions is
$$\frac1{z^2_n} - \frac1{z^2_n}=0.$$
This leads to a theorem in the spirit of \cite{DaHa}.

\begin{theorem}\label{DaHa Thm}
Suppose that \eqref{one} has two independent positive 
solutions for $m\le n\le N$, with $N\ge M+2$, and denote the 
associated Green matrix $G_{mn}$.
Since $G_{nn} > 0$ for $m\le n\le N$, we may
define $z_n := \sqrt{G_{nn}} > 0$.
In terms of $z_n$, determine
$S_n^{[z]}$
and
$\varphi^{\pm}_n$
according to
\eqref{Sfromz} and \eqref{phifromz}.
Then
\begin{enumerate}
\item $\varphi^{\pm}_n$ is an independent pair of solutions of \eqref{one}
for $m<n\le N$.
\item 
$G_{nm}=z_nz_m\prod^n_{\ell=m+1}\frac1{S_\ell^{[z]}}$, $M<m<n\le N$.
\item
The potential function is determined from $G_{nn}$ by a nonlinear difference
equation,
\begin{equation}\label{GtoV}
\frac{1}{2} \left(\sqrt{1 + 4 G_{n\,n}G_{n+1\,n+1}} + \sqrt{1 + 4 G_{n\,n}G_{n-1\,n-1}}\right) 
= (V_n+2) G_{nn}.
\end{equation}
\end{enumerate}
\end{theorem}

\begin{remark}\nonumber
In what follows we are mainly concerned with what happens when
$N\to\infty$.  In that case
the assumption that there are two positive solutions is 
a question of
disconjugacy in the theory of ordinary differential equations, 
cf.\ \cite{Hart,Aga}. 
If, for example, $V_n > 0$ for $n \ge N_0$, then it is 
not difficult to show that no solution can change sign more than once, and that therefore the positivity assumption is satisfied for $n$ sufficiently large.  
As will be seen in the proof, a necessary 
condition for the assumption is that  $V_n > -2$.

Per the symmetry remarked upon in \eqref{sym}, an alternative to
positivity is the assumption
that there are two solutions $\psi_n^\pm$ such that $(-1)^n \psi_n^\pm > 0$.
A sufficient condition for this is that $V_n < -4$ and a necessary condition is
that $V_n < -2$.
\end{remark}

%

\begin{proof} 
The essential calculation was provided 
in the discussion before the statement
of the theorem.  Given that the Wronskian of $\varphi^-$ and $\varphi^+$ is 1,
these two functions are linearly independent and therefore a basis for the 
solution space of
$$(-\Delta + V^{[z]}_n)\varphi=0,$$
\noindent
$V^{[z]}_n$ being defined by \eqref{Vfromz}.
Moreover,
$$G_{mn}=\varphi^+_{\min(m,n)}\varphi^-_{\max(m,n)}$$
is a Green function for $-\Delta + V^{[z]}_n$.

Hence the crux is to show that $V^{[z]}_n$ is the same as the original $V_n$ of
\eqref{one}.
Because $S_n^{[z]}$ was defined such that
\begin{equation*}
S_{n}^{[z]}-\frac1{S_{n}^{[z]}}=\frac{1}{z_nz_{n-1}},
\end{equation*}
we may rewrite \eqref{Vfromz} as
\begin{equation}\label{Vfromz alt}
V^{[z]}_n+2=
\frac{1}{2z^2_n}\left(
\sqrt{1+4 z_n^2 z_{n+1}^2}+
\sqrt{1+4 z_n^2 z_{n-1}^2}\right).
\end{equation}
From the definition of $z_n$ and the assumptions of the
theorem, we know that for some independent set of 
positive solutions 
$\psi_n^{\pm}$ of \eqref{one}, with Wronskian 1, $z_n^2 = \psi_n^+ \psi_n^-$.
Therefore
\begin{align}
4 z_n^2 z_{n \pm 1}^2 &= 4 (\psi_n^+ \psi_{n \pm 1}^-)(\psi_n^- \psi_{n \pm 1}^+)
\nonumber\\
&= (\psi_n^+ \psi_{n \pm 1}^- + \psi_n^- \psi_{n \pm 1}^+)^2 - (\psi_n^+ \psi_{n \pm 1}^- - \psi_n^- \psi_{n \pm 1}^+)^2
\nonumber\\
&=(\psi_n^+ \psi_{n \pm 1}^- + \psi_n^- \psi_{n \pm 1}^+)^2 - 1.\nonumber
\end{align}
Hence \eqref{Vfromz alt} yields
\begin{align}
V^{[z]}_n+2 &= \frac{1}{2 \psi_n^+ \psi_n^-} \left(\psi_n^+ \psi_{n+1}^- + \psi_n^- \psi_{n+1}^+ + \psi_n^+ \psi_{n-1}^- + \psi_n^- \psi_{n-1}^+ \right)\nonumber\\
&= \frac{1}{2 \psi_n^+ \psi_n^-} \left(\psi_n^+ V_n \psi_{n}^- + \psi_n^- V_n \psi_{n}^+ \right)\nonumber\\
&= V_n+2,\nonumber
\end{align} 
as claimed, and establishes \eqref{GtoV}.
\end{proof}

It may well be asked at this stage why we have restricted ourselves to the situation where 
$G_{nn} > 0$, for at the formal level the calculations given above remain valid without 
assuming positivity.  
In the discrete setting, continuity is not available to connect the values of a solution
$\varphi_n$ as $n$ varies, and hence 
without an assumption such as positivity,
there is a degree of indeterminateness in 
defining solutions by a prescription such as \eqref{Sfromz}.
For \emph{some} choice of phases in \eqref{Vfromz alt}, it will
still be true that $V_n^{[z]}$ as defined there 
coincides with $V_n$,
but the implicit nature of these choices of 
square root is problematic.
Possibly a suitable
canonical choice of phase or
ideas from 
Teschl's oscillation theory for Jacobi operators
\cite{Tes} could help avoid implicit
definitions, and we 
hope to elaborate this point in future work. 

Returning to the case where $G_{nn} > 0$,
Formula \eqref{phifromz} suggests that $S_n$ can be related to
an Agmon distance \cite{Agm,HiSi}, that is, a metric $d_A(m,n)$
on the positive integer lattice such that every $\ell^2$ solution 
$\phi^-$ of \eqref{one}
satisfies a bound of the form
$$e^{d_A(0,n)} \phi_n^- \in \ell^\infty,$$ 
and that as a consequence $\phi_n^-$ 
decays rapidly as $n \to \infty$.  Thus if
$z_n$ is bounded we expect
an Agmon distance to be something like
$\sum_{\ell=m+1}^n{\ln S_{\ell}^{[z]}}$, assuming $n > m$.
(We write the Agmon distance in this way because
the triangle inequality is an equality on the integer lattice, 
which implies that any metric  
takes the form of a sum of quantities defined at values of $\ell$ from $m+1$ to $n$.) 
In Agmon's theory, however, 
it is desirable that the distance function be a quantity that
can be calculated directly from the potential alone
(or at least dominated by some such expression).
As we shall now see, understanding the diagonal of the Green matrix
allows the derivation
of Agmonish bounds.
We begin
by showing that 
$G_{nn}$ is comparable to $(V_n+2)^{-1}$ in a precise sense.

\begin{lemma}\label{G comp 1/V}
Suppose that $\liminf_{n \to \infty}{V_n} > C > 0$, and let 
$G_{m n}$ be any positive Green matrix for \eqref{one}
on the positive integers.  
Define
$$
K_A := \sqrt{1+\left(\frac{2}{C(C+2)}\right)^2} +  \frac{2}{C(C+2)}.
$$
Then for $n$ sufficiently large,
\begin{equation}\label{Glowerupper}
\frac{1}{V_n+2} \le G_{nn} \le \frac{K_A}{V_n+2}.
\end{equation}
Consequently,
\begin{align}\label{Slowerupper}
&\frac{\sqrt{(V_n+2)(V_{n-1}+2)}+ \sqrt{4 +(V_n+2)(V_{n-1}+2)}}{2 K_A}
\le S_n^{[z]}\nonumber\\
&\qquad  \le 
\frac{\sqrt{(V_n+2)(V_{n-1}+2)}+ \sqrt{4 +(V_n+2)(V_{n-1}+2)}}{2}.
\end{align}
\end{lemma}

\begin{remark} The upper bound is of the same form as 
a semiclassical upper bound proved in \cite{HaWo}. To simplify it,
$K_A$ could be replaced in these inequalities by
$$
\sqrt{1 + \frac{4}{C^2}} > K_A
$$
(see proof).
\end{remark}

\begin{proof}
The lower bound on $G_{nn}$ is immediate from Statement (3) of
Theorem \ref{DaHa Thm}, the left member of which is larger than 1.
\medskip

\noindent
The upper bound in \eqref{Glowerupper}
requires a spectral estimate.  The Green 
matrix $G_{mn}$ is the kernel of the resolvent operator of a self-adjoint realization of
$- \Delta + V$ on $\ell^2([N, \infty))$ for some $N$, 
where the boundary condition at $n=N,N+1$ is that satisfied by $\varphi_n^+$.  
Since $- \Delta > 0$ on this space (as an operator),
$\inf {\rm sp}(- \Delta + V) > C$, and hence, by the spectral mapping theorem,
$\|(- \Delta + V)^{-1}\|_{\rm op} < C^{-1}$.  Since 
$G_{nn} = \left\langle{e_n, (- \Delta + V)^{-1} e_n }\right\rangle$, 
where $\{e_n\}$ designate the standard unit vectors in $\ell^2$, it follows that
$G_{nn} < C^{-1}$.  Inserting this into
\eqref{Vfromz} would already imply \eqref{Glowerupper} with
$K_A$ replaced by $\sqrt{1+4/C^2}$.  To improve the constant, replace only 
the terms $G_{n\pm1\,n\pm1}$ in \eqref{Vfromz alt} by $1/C$, getting
\begin{equation}\label{quadrat}
(V_n+2) \le \frac{\sqrt{1 + \frac{4G_{nn}}{C}}}{G_{nn}}.
\end{equation}
Since
$$
\frac{\sqrt{1 + x y}}{x}
$$
is a decreasing function of $x$ when $x,y > 0$, an upper bound on
$G_{nn}$ is the larger root of the case of equality in \eqref{quadrat}
(which is effectively a quadratic).
The claimed upper bound with the constant $K_A$ results 
by keeping one factor $V_n+2$ in the solution of the quadratic, replacing 
the others by $C+2$.

\noindent
The bounds on $S_n^{[z]}$ result from inserting the bounds on $G_{nn}$ into
\eqref{Sfromz} and collecting terms.
\end{proof}

We can now state some Agmonish bounds.
\begin{cor}\label{Agmon}
Suppose that $\liminf_{n \to \infty}{V_n} > C > 0$ and fix a positive integer $m$.  Then 
the subdominant (i.e., eventually decreasing) 
solution $\phi^-$ of \eqref{one} satisfies
\item{(a)}
$$
\left(\prod_{\ell=m}^n{\frac{V_\ell+2}{K_A}}\right)\phi_n^-  \in \ell^\infty.
$$
\item{(b)}  If, in addition, $n({V_{n+1}-V_n}) \in \ell^1$, then
$$
\left(\prod_{\ell=m}^n{\frac{V_\ell+2 + \sqrt{V_\ell(V_\ell + 4)}}{2}}\right) \phi_n^-  \in \ell^\infty.
$$
\end{cor}

\begin{proof}
The ansatz \eqref{Sfromz} allows an identification of 
$\phi^-$ with a constant multiple of $\varphi^-$,
in the representation
\eqref{phifromz}.  
Because $z_n$ is bounded, so is
$$
\left(\prod_{\ell}^n{S_\ell^{[z]}}\right) \varphi_n^-.
$$
We then use the lower bound on $S_\ell^{[z]}$ from the lemma, but 
simplify by dropping the 4, which allows the product to telescope in a pleasing way,
producing (a).

For (b) we note that the additional assumption on $V_n$ allows us to conclude that 
$\varphi$ is well-approximated by a Liouville-Green expression in 
\cite{HaWo}, Theorem 4.1, which is a bounded quantity times the reciprocal of the expression in parentheses.
\end{proof}

\noindent
Thus when
$\liminf_{n \to \infty}{V_n} > 0$,
a suitable Agmon distance $d_A(m,n)$ for \eqref{one}
is given by 
$$
\sum_{\ell=m+1}^n{\left(\ln(V_l +2) - \ln{K_A}\right)},
$$
or by 
$$
\sum_{\ell=m+1}^n{\ln{\frac{V_\ell+2 + \sqrt{V_\ell(V_\ell + 4)}}{2}}},
$$
provided that 
$n(V_{n+1}-V_n) \in \ell^1$.

We close with a Darboux-type factorization for a generic discrete Schr\"odinger equation
\eqref{one}.  A Darboux-type factorization for general 
Jacobi operators was previously considered by Gesztesy and Teschl in
\cite{GeTe2}.  As in \S\ref{B&G} the novel feature of this factorization
is that it is constructed using the diagonal 
of the Green matrix.

To this end, choose a Green matrix such that $G_{nn}$ is nonvanishing for a range of values of $n$,
and define a solution $\varphi_n^+$ according to \eqref{Sfromz}.  The phase of the
square roots is chosen (if necessary) to ensure that $V_n^{[z]}$ from 
\eqref{Vfromz alt} equals $V_n$.
\begin{theorem}
Given a Green matrix such that the diagonal $G_{kk}$ is nonvanishing for
$n-1 \le k \le n+2$, and choosing the phase of
the square roots as described above,
\begin{equation*}\label{discDarb}
- \Delta + V_n = \mathcal{R} \left[- \nabla^+ -1 +\frac{2 G_{nn}}{1+(1+4 G_{nn} G_{n+1\, n+1})^{1/2}}\right]  \left[{\nabla^+ + 1 - \frac{1+(1+4 G_{nn} G_{n+1\, n+1})^{1/2}}{2 G_{nn}}}\right],
\end{equation*}
where $\mathcal{R}$ is the shift operator such that $[\mathcal{R} f]_n = f_{n-1}$
and the right-difference operator is defined by
$\left[\nabla^{+} f\right]_n := f_{n+1} - f_n$.
\end{theorem}

\begin{remark}
As with \eqref{Darboux}, there is a second factorization, with shifts and differences reversed, 
and $n+1$ replaced by $n-1$.
\end{remark}

\begin{proof}
Writing
$$
Q_n := 1 - \frac{1+(1+4 G_{nn} G_{n+1\, n+1})^{1/2}}{2 G_{nn}} =1 - \frac{z_{n+1} S_{n+1}^{[z]}}{z_n},
$$
with $S_{k}^{[z]}$defined in \eqref{Sfromz}, 
we first note that, by a simple calculation,
\begin{equation}\label{1st order}
\left[\nabla^+  +Q_n \right]\varphi^+ = 0.
\end{equation}
This motivates calculating
$$
H = \left[- \nabla^+ +\frac{Q_n}{1-Q_n}\right] \left[\nabla^+ +Q_n\right],
$$
which is well-defined because $Q_n \ne 1$, owing to \eqref{Sfromz}
with $z_k$ nonvanishing.  The left factor was chosen to produce 
a convenient cancellation, ensuring that
$H$ has the form of a discrete Schr\"odinger equation, 
with a shifted index:
$$
(H f)_n = (- \Delta f)_{n+1} + \left(\frac{Q_n}{1-Q_n} - Q_{n+1}\right) f_{n+1} + 0\cdot f_n.
$$
We now verify that when the index is shifted back, the potential term is indeed $V_n$:
\begin{align*}
\left(\frac{Q_{n-1}}{1-Q_{n-1}} - Q_{n}\right) &= \left(\frac{z_{n-1}}{z_n S_n^{[z]}} -1\right) -
\left(1 - \frac{z_{n+1}S_{n+1}^{[z]}}{z_n}\right)\\
&= -2 + \frac{z_{n-1}}{z_n S_n^{[z]}} + \frac{z_{n+1}S_{n+1}^{[z]}}{z_n},
\end{align*}
which reduces to $V_n$ according to
\eqref{Vfromz}.
\end{proof}

\end{document}